\documentclass[10pt,draft,twoside]{amsart}
\usepackage{amssymb}
\usepackage{latexsym}
\usepackage{amsfonts}
\usepackage{amsmath}
\usepackage{hyperref}

\DeclareMathOperator{\Aut}{Aut}
\DeclareMathOperator{\PSL}{PSL}

\DeclareMathOperator{\supp}{supp}
\DeclareMathOperator{\pt}{\mathcal{P}}

\DeclareMathOperator{\B}{\mathcal{B}}

\DeclareMathOperator{\Diag}{Diag}

\DeclareMathOperator{\Mon}{Mon}

\newtheorem{theorem}{Theorem}[section]
\newtheorem{proposition}[theorem]{Proposition}
\newtheorem{lemma}[theorem]{Lemma}

\theoremstyle{definition}

\newtheorem{remark}[theorem]{Remark}
\newtheorem{definition}[theorem]{Definition}

\newcommand{\F}{\mathbb F}

\renewcommand{\leq}{\leqslant}
\renewcommand{\geq}{\geqslant}

\numberwithin{equation}{section}

\begin{document}

\title[Uniqueness of Certain Completely Regular
 Hadamard Codes] {Uniqueness of Certain\\ Completely Regular Hadamard Codes}
\author{Neil I. Gillespie and Cheryl E. Praeger}
\address{[Gillespie and Praeger] Centre for Mathematics of Symmetry and
Computation\\
School of Mathematics and Statistics\\
The University of Western Australia\\
35 Stirling Highway, Crawley\\
Western Australia 6009}

\email{neil.gillespie@graduate.uwa.edu.au, cheryl.praeger@uwa.edu.au}

\begin{abstract}
We classify binary completely regular codes of length
$m$ with minimum distance $\delta$ for $(m,\delta)=(12,6)$ and
$(11,5)$.  We prove that such codes are unique up to equivalence, and in
particular, are equivalent to certain Hadamard codes.  Moreover, we prove
that these codes are completely transitive.         
\end{abstract}

\thanks{This paper is dedicated to Kathryn Horadam in honour of her 60th
birthday.\\
{\it Date:} draft typeset \today
This research was supported by the Australian Research Council Federation Fellowship
FF0776186 of the second author, and for the first author, by an Australian Postgraduate Award.} 

\maketitle

\section{Introduction}\label{s1}

In 1973 Delsarte \cite{delsarte} introduced a class of codes with a high
degree of combinatorial symmetry called \emph{completely regular
  codes} (see Definition \ref{cregdef}).  Examples of such codes
include the Hamming, Golay, Preparata, and Kasami codes.  However, a belief
began to emerge amongst coding
theorists that completely regular codes with  
good error correcting capabilities are rare \cite{martin}.  Indeed, in 1992,
Neumaier \cite{neum} conjectured that the only completely regular
code containing more than two codewords with minimum distance at least $8$ is
the extended binary Golay code.  
Despite this, there has been renewed interest in the subject.  Borges, Rif{\`a}
and Zinoviev have
written a series of papers classifying various families and finding
new examples of completely regular codes \cite{rho=2,
  binctrarb, kronprod, extendconst}, and in particular, they give a 
counter example to Neumaier's conjecture \cite{nonantipodal}. Moreover, they also show that many of these examples 
are \emph{completely transitive} (see Definition \ref{cregdef}), a subclass
of completely regular codes. Given these new discoveries, we consider
the following two problems posed by Neumaier in his 1992 
paper.  \begin{itemize}\item Find all `small' completely regular codes
  in Hamming graphs.
\item Characterise the known completely regular codes by their
  parameters.
\end{itemize}  In relation to these two questions we draw attention
to two beautiful examples of binary completely regular codes;
\emph{the Hadamard $12$ code} and \emph{the punctured Hadamard $12$
  code} (see Section \ref{sechad}). In this paper we characterise
these codes by their parameters and prove the following.   

\begin{theorem}\label{mainchhad2} Let $C$ be a binary completely regular code of
length $m$
  with minimum distance $\delta$. 
\begin{itemize} 
\item[(a)]If $m=12$ and $\delta=6$, then $C$ is equivalent to the
  Hadamard $12$ code; 
\item[(b)] If $m=11$ and $\delta=5$, then $C$ is equivalent to the
  punctured Hadamard $12$ code.
\end{itemize}
Moreover the codes $C$ in \textup{(a)} and \textup{(b)} are completely
transitive.
\end{theorem}

In Section \ref{prelim}, we introduce our notation and preliminary
results.  The Hadamard $12$ code is an example of a \emph{Hadamard
  code}, a family of codes generated by Hadamard matrices, which we
describe in Section \ref{autgphad}, and prove that the automorphism
group of a matrix is isomorphic to the automorphism group of the
code it generates.  In the subsequent section we prove that the
Hadamard $12$ and punctured Hadamard $12$ codes are completely
transitive.  In the final section we prove Theorem \ref{mainchhad2}.     

\section{Definitions and Preliminaries}\label{prelim}

Any binary code of length $m$ can be embedded as a subset of the vertex set of
the binary 
Hamming graph $\Gamma=H(m,2)$, which has a vertex set $V(\Gamma)$ that consists
of $m$-tuples with
entries from the field $\F_2=\{0,1\}$, and an edge exists between two
vertices if and only if they differ in precisely one entry.  The
automorphism group of $H(m,2)$, which we denote by $\Aut(\Gamma)$, is
the semi-direct product $B\rtimes L$ where $B\cong S_2^m$ and $L\cong
S_m$ \cite[Thm. 9.2.1]{distreg}.  Let $g=(g_1,\ldots, g_m)\in
B$, $\sigma\in L$ and $\alpha=(\alpha_1,\ldots,\alpha_m)\in V(\Gamma)$. 
Then $g\sigma$ acts on $\alpha$ in the following way:        
\begin{align*}
\alpha^{g\sigma}=&(\alpha_{1^{\sigma^{-1}}}^{g_{1^{\sigma^{-1}}}},\ldots,\alpha_
{m^{\sigma^{-1}}}^{g_{m^{\sigma^{-1}}}}).
\end{align*} Let $M=\{1,\ldots,m\}$.  The \emph{support of $\alpha$}
is the set $\supp(\alpha)=\{i\in M\,:\,\alpha_i\neq 0\}$, and the
\emph{weight of $\alpha$} is the size of the set $\supp(\alpha)$.  For all
pairs of vertices $\alpha,\beta\in V(\Gamma)$, the \emph{Hamming
  distance} between $\alpha$ and $\beta$, denoted by
$d(\alpha,\beta)$, is defined to be the number of entries in which the
two vertices differ.  We let $\Gamma_k(\alpha)$ denote the set of
vertices in $\Gamma$ that are at distance $k$ from $\alpha$.       

A code $C$ in $H(m,2)$ is an $(m,N,\delta)$ code if $N=|C|$, the number of
codewords of $C$, and $\delta$ 
is the \emph{minimum distance of $C$}, the smallest distance between distinct
codewords of $C$.  For
any vertex $\gamma\in V(\Gamma)$, we define the \emph{distance of
  $\gamma$ from $C$} 
to be $d(\gamma,C)=\min\{d(\gamma,\beta)\,:\,\beta\in C\}.$  The 
\emph{covering radius $\rho$ of $C$} is the maximum distance any vertex in
$H(m,2)$ is from $C$.  We let $C_i$ denote the set of vertices that
are distance $i$ from $C$.  It follows that $\{C=C_0, C_1, \ldots,C_\rho\}$
forms
a partition of $V(\Gamma)$, called the \emph{distance
  partition of $C$}.  Let $C'$ be another code in $H(m,2)$.  We say
$C$ and $C'$ are \emph{equivalent} if there exists
$x\in\Aut(\Gamma)$ such that $C^x=C'$, and if $C=C'$, we call $x$ an
\emph{automorphism of $C$}.  The \emph{automorphism group of $C$}
is the setwise stabiliser of $C$ in $\Aut(\Gamma)$, which we denote by
$\Aut(C)$.  The \emph{distance distribution of $C$} is the 
$(m+1)$-tuple $a(C)=(a_0,\ldots,a_m)$
where \begin{equation}\label{ai}a_i=\frac{|\{(\alpha,\beta)\in 
  C^2\,:\,d(\alpha,\beta)=i\}|}{|C|}.\end{equation}   
We observe that $a_i\geq 0$ for all $i$ and $a_0=1$.  Moreover,
$a_i=0$ for $1\leq i\leq \delta-1$ and $|C|=\sum_{i=0}^ma_i$. 
The \emph{MacWilliams transform} of $a(C)$ is the $(m+1)$-tuple 
$a'(C)=(a_0',\ldots,a_m')$
where \begin{equation}\label{kracheqn}a'_k:=\sum_{i=0}^ma_iK_k(i)\end{equation}
with \begin{equation*}K_k(x):=\sum_{j=0}^k(-1)^j\binom{x}{j}\binom{m-x}{k-j}.\end{equation*}
It follows from \cite[Lem. 5.3.3]{vanlint} that $a'_k\geq 0$ for $k=0,1,\ldots,m$.  We say $C$ has
\emph{external distance $s$} if there are exactly $s+1$ nonzero entries in $a'(C)$.    
 
\begin{definition}\label{cregdef}  Let $C$ be code with distance partition
$\{C=C_0, C_1,\ldots, C_\rho\}$
 and $\gamma\in C_i$.  We say $C$ is \emph{completely regular} if
$|\Gamma_k(\gamma)\cap C|$ depends
 only on $i$ and $k$, and not on the choice of $\gamma\in C_i$.  If there exists
$G\leq \Aut(\Gamma)$ 
  such that each $C_i$ is a $G$-orbit, then we say $C$ is \emph{$G$-completely
transitive}, or simply \emph{completely transitive}.
\end{definition}

A narrower definition of complete transitivity for binary linear codes was first
introduced by Sol{\'e} \cite{sole}.  Giudici and the second author \cite{giupra}
extended Sol{\'e}'s definition to cover linear codes over any finite
field, referring to such codes as \emph{coset-completely transitive}.
(This is the definition used by Borges et al. in their series of
papers.)  In the same paper, Giudici and the second author further
generalised this concept. (The definition given here is the one given
in \cite{giupra}, but in the binary case.)  They proved that
coset-completely transitive codes are necessarily completely 
transitive, and that completely transitive codes are completely
regular.  Moreover, they showed that the repetition code of length
$3$, over an alphabet with large enough cardinality, is completely
transitive but not coset-completely transitive.  We give two
additional examples in this paper (Proposition \ref{had12prop} and
\ref{had12pun}), which are binary and non-trivial.  The following
result is straight forward, but useful in Section \ref{sechad} where we prove
the complete transitivity of these examples.  

\begin{lemma}\label{method} Let $C$ be a code in $H(m,2)$ such that
  $\Aut(C)$ acts transitively on $C$.  If $\Aut(C)_\alpha$ acts
  transitively on $\Gamma_i(\alpha)\cap C_i$ for some $\alpha\in C$,
  then $\Aut(C)$ acts transitively on $C_i$. 
\end{lemma}  

A code $C$ with covering radius $\rho$ and minimum distance $\delta$ is \emph{uniformly packed (in the wide
sense)}, if there exist rational numbers $\lambda_0,\ldots,\lambda_\rho$ such that 
$\sum_{i=0}^{\rho} \lambda_i f_k(\nu)=1$ for all vertices $\nu$, where
$f_k(\nu)=|\Gamma_k(\nu)\cap C|$.  The case $\rho=e+1$, with $e=\lfloor\frac{\delta-1}{2}\rfloor$, 
corresponds to the \emph{uniformly packed codes} defined by Goethals and van Tilborg \cite{ufpc1}.
For any code $C$ of length $m$, the \emph{extended code $C^*$} of length $m+1$ is generated
by adding a parity check bit to each codeword of $C$.  In the case that $C$ is a 
binary uniformly packed code, Bassalygo and Zinoviev \cite{remufp} gave necessary and
sufficient conditions for $C^*$ to be uniformly packed.  The following result seems to be well known, 
but the authors could not find a proof in the literature, so we give a simple argument here.

\begin{lemma}\label{ufplem} Let $C$ be a binary uniformly packed code with covering radius $\rho=e+1$,
where $e=\lfloor\frac{\delta-1}{2}\rfloor$, such that the extended code $C^*$ is uniformly packed.
Then $C^*$ is completely regular.
\end{lemma}

\begin{proof}  It is straight forward to deduce that $C^*$ has minimum distance $\delta^*=2e+2$.  
Moreover, by \cite{brou}, $C^*$ has 
covering radius $\rho^*=\rho+1$.  As $C^*$ is uniformly packed, it follows from \cite{remufp} that it 
has external distance $s^*=\rho^*=\rho+1=e+2$.  Consequently, $2s^*-2=2e+2=\delta^*$, and because $C^*$ consists
entirely of codewords of even weight, it follows from \cite[p.347]{distreg} that $C^*$ is completely regular.
\end{proof}

For any $\alpha\in V(\Gamma)$, the \emph{complement of $\alpha$} is the unique vertex
$\alpha^c$ with $\supp(\alpha^c)=M\backslash\supp(\alpha)$, and we say a code $C$
is \emph{antipodal} if $\alpha^c\in C$ for all $\alpha\in C$.

Let $p\in M=\{1,\ldots m\}$ and $C$ be a code in $H(m,2)$.  By
deleting the same coordinate $p$ from each codeword of $C$, we
generate a code $C^{(p)}$ in $H(m-1,2)$, which we call the \emph{punctured
  code of $C$ with respect to $p$}.  We can also describe $C^{(p)}$ in the 
following way.  Let 
$J=\{i_1,\ldots,i_k\}\subseteq M$ and define the following   
map \[\begin{array}{c c c c} \pi_J:&H(m,2)&\longrightarrow&H(J,2)\\  
&(\alpha_1,\ldots,\alpha_m)&\longmapsto&(\alpha_{i_1},\ldots,\alpha_{i_k})\\
\end{array}\]  Note that we can identify $H(J,2)$ with $H(|J|,2)$.  We
define the \emph{projected code of $C$ with respect to $J$} to be the
set $\pi_J(C)=\{\pi_J(\alpha)\,:\,\alpha\in C\}$.  It follows that if 
$J=M\backslash\{p\}$ then $\pi_J(C)=C^{(p)}$.  We have an induced action of
$\Aut(\Gamma)_J=\{g\sigma\in 
\Aut(\Gamma)\,:\,J^\sigma=J\}$ on $H(J,2)$ as follows: for $x\in
\Aut(\Gamma)_J$,
we define $\chi(x):H(J,q)\longrightarrow H(J,q)$ by
$\pi_J(\alpha)\longmapsto\pi_J(\alpha^x)$, and obtain
$\chi(\Aut(\Gamma)_J)\leq\Aut(H(J,2))$ with
$\ker\chi=\{(g_1,\ldots,g_m)\sigma\in\Aut(\Gamma)_J\,:\,j^\sigma=j\textnormal{
  and }g_j=1\textnormal{ for }j\in J\}.$     

Let $\mathcal{D}=(\pt,\B)$ where $\pt$ is a set of points of
cardinality $m$, and $\B$ is a set of $k$-subsets of $\pt$ called
blocks.  Then $\mathcal{D}$ is a \emph{$t-(m,k,\lambda)$ design} if
every $t$-subset of $\pt$ is contained in exactly $\lambda$ blocks of
$\B$.  We let $b$ denote the number of blocks in $\mathcal{D}$.  If 
$\mathcal{D}$ is a $t$-design, then it is also an $i-(m,k,\lambda_i)$
design for $0\leq i\leq t-1$ 
\cite[Corollary 1.6]{camvan}
where
\begin{equation}\label{arith1}\lambda_i\binom{k-i}{t-i}=\lambda\binom{m-i}{t-i}
.\end{equation}
Using this fact we deduce
that
\begin{equation}\label{arith2}\binom{m}{i}\lambda_i=b\binom{k}{i}.\end{equation}
An \emph{automorphism of $\mathcal{D}$} is a permutation of
$\mathcal{P}$ that preserves $\mathcal{B}$, and we let
$\Aut(\mathcal{D})$ denote the group of automorphisms of
$\mathcal{D}$.  For further concepts and definitions about $t$-designs
see \cite{camvan}.    

\begin{remark}\label{desrem} Two designs that appear in this paper are
  $\mathcal{D}$, a $2-(11,5,2)$ design, and $\mathcal{E}$, a $3-(12,6,2)$ 
  design.  Todd proved that $\mathcal{D}$ is the unique design with
  these parameters up to isomorphism, and that
  $\Aut(\mathcal{D})\cong\PSL(2,11)$ \cite{todd}.  Using this 
  we can prove that $\mathcal{E}$ is unique up to isomorphism (see \cite[Ch.
2]{ngthesis}, for example).  Furthermore, it 
  is known that $\Aut(\mathcal{E})\cong M_{11}$, acting
  $3$-transitively on $12$ points \cite{hughes}.     
\end{remark}

For $\alpha$, $\beta\in V(\Gamma)$, we say
$\alpha$ is \emph{covered} by $\beta$ if for each non-zero 
component $\alpha_i$ of $\alpha$, we have $\alpha_i=\beta_i$.  Let
$\mathcal{D}$ be a set of vertices of weight $k$ in $V(\Gamma)$ and 
$t\leq k$.  We say $\mathcal{D}$ is a \emph{$2$-ary
  $t$-$(m,k,\lambda)$ design} if for every vertex of weight $t$ in
$V(\Gamma)$ is covered by exactly $\lambda$ vertices of $\mathcal{D}$.
This definition coincides with the usual definition of a $t$-design,
in the sense that the set of blocks of the $t$-design is the set of
supports of vertices in $\mathcal{D}$, and as such we simply refer to $2$-ary
designs as 
$t$-designs.  It is known that, for a completely regular code $C$ in $H(m,2)$ with 
zero codeword and minimum distance $\delta$, the set $C(k)$ of codewords of weight $k$ 
forms a $t-(m,k,\lambda)$ design for some $\lambda$ with 
$t=\lfloor\frac{\delta}{2}\rfloor$ \cite{ufpc1}.        

\section{Hadamard Codes and their Automorphism Groups}\label{autgphad}

A \emph{Hadamard matrix} of order $m$ is an $m\times m$ matrix $H$
with entries $\pm 1$ satisfying $HH^T=mI$.  We denote 
the $i$th row of a Hadamard matrix $H$ by $r_i$.  A consequence of 
the definition of $H$ is that $(r_i\cdot r_j)=m\delta_{ij}$, where $(-
\cdot -)$ denotes the standard dot product and $\delta_{ij}$ is the
Kronecker delta function.  Any two Hadamard matrices $H_1$, $H_2$ of
order $m$ are said to be \emph{equivalent} if there exist $m\times m$
monomial matrices $P$, $U$, with nonzero entries $\pm 1$, such that
$H_1=PH_2U$.  Let $\Mon_m(\pm 1)$ denote the group of $m\times m$
monomial matrices with nonzero entries $\pm 1$.   For $P,U\in
\Mon_m(\pm 1)$, we let $\varphi_{P,U}$ denote the map $H\mapsto PHU$,
and call $\varphi_{P,U}$ an \emph{automorphism of $H$} if $H=PHU$.  We
denote the group of automorphisms of $H$ by $\Aut(H)$.  

The Hadamard codes were first introduced by Bose and Shrikhande
\cite{codeconstruct}, and 
studied further in \cite{remufp, ufpc1}.  Our treatment follows that in
\cite{codeconstruct, ufpc1}.
Let $V_m(\pm 1)=\{(a_1,\ldots,a_m)\,:\,a_i=\pm 1\}$ and define the following
bijection: \begin{equation*}\begin{array}{ccc}\begin{array}{ccc} 
      \kappa: V_m(\pm
      1)&\longrightarrow&H(m,2)\\
(a_1,\ldots,a_m)&\longmapsto&(\alpha_1,\ldots,\alpha_m)
\end{array}&\textnormal{where}&\alpha_i=\left\{\begin{array}{cl}
    1&\textnormal{if $a_i=-1$}\\ 0&\textnormal{if
      $a_i=1$.} \end{array}\right.\end{array}\end{equation*}
The \emph{Hadamard code}
  generated by a Hadamard matrix $H$ of order $m$ is defined as $$C(H):=\{\kappa(r)\,:\,\textnormal{$r$ is a
  row of $H$ or $-H$}\}.$$ Bose and Shrikhande showed that $C(H)$ is an $(m,2m,\frac{1}{2}m)$ code.
  Also, we note that since $\kappa(-r)\in
C(H)$ for all $\kappa(r)\in C(H)$, it
follows that $C(H)$ is antipodal.  Let $\tau$ be the
monomorphism \begin{equation}\label{tau}\begin{array}{ccc}      
\tau: \Aut(H)&\longrightarrow&\Mon_m(\pm 1)\\
\varphi_{P,U}&\longmapsto&U
\end{array}\end{equation}  Any $U\in \Mon_m(\pm 1)$ can be described
uniquely as a product of a diagonal matrix
$U_D=\Diag(u_1,\ldots,u_m)$, with $u_i=\pm 1$, 
followed by a permutation matrix $U_\sigma=(u_{ij})$, where $\sigma\in
S_m$ and $u_{ij}=0$ if $j\neq i^\sigma$, $u_{i,i^\sigma}=1$.  If
$\Gamma=H(m,2)$ it follows that 
$\Mon_m(\pm 1)\cong\Aut(\Gamma)$ via the
isomorphism  \begin{equation}\label{theta}\begin{array}{rcl} 
\theta: \Mon_m(\pm 1)&\longrightarrow&\Aut(\Gamma)\\
U=U_DU_\sigma&\longmapsto&(\theta^*(u_1),\ldots,\theta^*(u_m))\sigma, 
\end{array}\end{equation} where $\theta^*$ is the unique isomorphism
from the multiplicative group $Z=\{1,-1\}$ to $S_2$.  With respect to
$\theta$ and $\kappa$, the action of $\Mon_m(\pm 1)$ on $V_m(\pm 1)$ 
is permutationally isomorphic to $\Aut(\Gamma)$ acting on $V(\Gamma)$.
That is \begin{equation}\label{permiso} 
  \kappa(vU)=\kappa(v)^{U\theta},\,\,\,\,\forall\,U\in \Mon_m(\pm
1),\,\forall\,v\in
  V_m(\pm 1).\end{equation}  

\begin{proposition}\label{autchad} Let $H$ be a Hadamard
  matrix of order $m$.  Then $\tau\theta$ is an isomorphism from
  $\Aut(H)$ to $\Aut(C(H))$, where $\tau$, $\theta$ are as in
  (\ref{tau}), (\ref{theta}) respectively.      
\end{proposition}

\begin{proof}  Let $P=P_DP_\sigma,U\in\Mon_m(\pm 1)$ and
  $\overline{H}:=PHU$.  Let $\alpha\in C(H)$, so
  $\alpha=\kappa(\lambda r_i)$ for some row $r_i$ of
  $H$ and $\lambda=\pm 1$. By 
  (\ref{permiso}), it holds that $\kappa(\lambda
  r_i)^{U\theta}=\kappa(\lambda r_iU)$.  Since
  $P^{-1}\overline{H}=HU$, we deduce that $\lambda
  r_iU=\bar{\lambda}\bar{r}_{i^{\sigma^{-1}}}$, where
  $\bar{r}_{i^{\sigma^{-1}}}$ is the $i^{\sigma^{-1}}$ row of
  $\overline{H}$ and $\bar{\lambda}=\pm 1$.  Thus
  $\alpha^{U\theta}=\kappa(\bar{\lambda} \bar{r}_{i^{\sigma^{-1}}})\in
C(\overline{H})$.  As
  $C(H)$ and $C(\overline{H})$ have the same 
  cardinality, it follows that $C(H)^{U\theta}=C(\overline{H})=C(PHU)$.  In
  particular, equivalent Hadamard matrices generate equivalent codes,
  and if $\varphi_{P,U}\in\Aut(H)$ then $U\theta\in\Aut(C)$.  That is,
  $(\Aut(H))\tau\theta\leq\Aut(C(H))$. 

  Let $x\in \Aut(C(H))$ and set $U:=x\theta^{-1}\in
  \Mon_m(\pm 1)$.  Let $P:=HU^{-1}H^{-1}$, so $PHU=H$, and since
  $H^{-1}=\frac{1}{m}H^T$, it follows that
  $P_{ik}=\frac{1}{m}(r_iU^{-1}\cdot r_k)$.  We claim that $P$ is a
  monomial matrix, and so $\varphi_{P,U}\in \Aut(H)$ and
  $(\varphi_{P,U})\tau\theta=x$.  Since $U\theta\in \Aut(C(H))$, we
  deduce that $\kappa(r_iU^{-1})=\kappa(r_i)^{U^{-1}\theta}\in C(H)$
  for each $i$.  In particular, there exist $r_{j_i}$, $\lambda_i=\pm
  1$ such that $r_iU^{-1}=\lambda_i r_{j_i}$ for each $i$.  As $C(H)$ is 
  antipodal, it follows that $\Aut(C(H))$ acts on the set of pairs of
  complementary codewords in $C(H)$, from which we deduce that $U^{-1}$
  maps $\{r_1,\ldots,r_m\}$ to a set of pairwise orthogonal vectors.
  Thus, for each $i$, there exists a unique $k$ such that
  $(r_{j_i}\cdot r_k)\neq 0$, and because these are rows of a Hadamard
  matrix it follows that $(r_{j_i}\cdot r_k)=m\delta_{j_ik}$.  Thus
  $P_{ik}=\lambda_i\delta_{j_ik}$, so $P\in \Mon_m(\pm 1)$ and the
  claim holds.
\end{proof}

\section{The Hadamard $12$ and Punctured Hadamard $12$ Codes}\label{sechad}

In this section we consider Hadamard codes generated by Hadamard
matrices of order $12$.  It is known that a Hadamard matrix of 
order $12$ is unique up to equivalence \cite{hall}, and we saw in the 
proof of Proposition \ref{autchad} that equivalent Hadamard matrices
generate equivalent Hadamard codes.  Therefore we can choose a normalised
Hadamard
matrix $H_{12}$ of order $12$, namely one in which the first row and first
column
have all entries equal to $1$.  We call the Hadamard code $C(H_{12})$ 
the \emph{Hadamard $12$ code}.  Throughout this section $C$ 
denotes the Hadamard $12$ code, and $C^{(1)}$ denotes the \emph{punctured Hadamard
$12$ code} generated by deleting the first
entry of each codeword of $C$.  By \cite{codeconstruct}, $C$ is a
$(12,24,6)$ code and $C^{(1)}$ is an $(11,24,5)$ code.  Van Lint \cite{vanlint}
showed
that $C^{(1)}$ is uniformly packed in the sense of \cite{ufpc1} with covering
radius $\rho^{(1)}=3$.
Thus, by \cite[Cor 12.2]{ufpc1}, $C^{(1)}$ is completely regular.  Because $H_{12}$ is normalised, it follows
that the extended code of $C^{(1)}$ with the parity check bit placed at the front of
each codeword is equal to $C$.  A result by Bassalygo and Zinoviev \cite{remufp} 
therefore implies that $C$ is uniformly packed (in the wide sense), and so by Lemma \ref{ufplem}, 
$C$ is completely regular with covering radius $\rho=4$.  

The code $C$ consists of the zero
vertex, the all $1$ vertex and $22$ vertices of weight $6$.  As $C$ is
completely regular, 
it follows from \cite{ufpc1} that $C(6)$, the set of codewords in $C$ of weight
$6$, forms a $3-(12,6,2)$ design.
In this design, the number of blocks that contain the first entry is
$r=11$.  Thus $C^{(1)}$ consists of $11$ codewords of
weight $5$ and their complements, as well as the zero and all $1$ vertices.  It
follows from 
\cite{codeconstruct} (or from the fact $C^{(1)}$ is completely
regular) that $C^{(1)}(5)$ forms a $2-(11,5,2)$ design.  Hall \cite{hall} showed
that $\Aut(H_{12})$ is 
the (non-split) double cover $2M_{12}$, and therefore, by Proposition
\ref{autchad}, $\Aut(C)\cong 2M_{12}$.  

\begin{proposition}\label{had12prop} The Hadamard $12$ code $C$ is
  $\Aut(C)$-completely transitive.  
\end{proposition}

\begin{proof}  Let $\mathcal{E}$ be the $3-(12,6,2)$
  design with block set $C(6)$ and $\alpha$ be the zero codeword.  
  Any automorphism $\sigma\in\Aut(C)_\alpha=\Aut(C)\cap L$ naturally
  induces an automorphism of $\mathcal{E}$, and similarly the converse
  holds.  Thus $\Aut(C)_\alpha\cong\Aut(\mathcal{E})\cong M_{11}$
  acting $3$-transitively on $12$ points \cite{hughes}, which is a subgroup of
index $24$
  in $\Aut(C)\cong 2M_{12}$.  Hence $\Aut(C)$ acts transitively on $C$.  Thus,
because 
  $\Aut(C)_\alpha$ acts transitively on $\Gamma_i(\alpha)$ for $i=1,2,3$, and
since $\delta=6$, 
  it follows that $\Aut(C)$ acts transitively on $C_i$ for $i=1,2,3$.    

  Let $M=\{1,\ldots,12\}$ and $\tilde{\nu}\in\Gamma_4(\alpha)$.  Then
  $\tilde{\nu}$ is adjacent to a weight $3$ vertex, which because
  $\delta=6$ is an element of $C_3$.  Thus $\tilde{\nu}\in C_2\cup
  C_3\cup C_4$.  Suppose $\tilde{\nu}\in C_3$.  Then there exists a codeword
  $\beta$ such that $d(\beta,\tilde{\nu})=3$.  Because $\tilde{\nu}$
  has  weight $4$ we deduce that $\beta$ has odd weight.  However,
  $C$ consists of codewords of only even weight.  Therefore
  $\tilde{\nu}\in C_2\cup C_4$.  For any codeword $\beta$ of weight
  $6$, there exist $\binom{6}{4}$ vertices of weight $4$ that are at
  distance $2$ from $\beta$.  Thus $\Gamma_4(\alpha)\cap
  C_2\neq\emptyset$.  Additionally, because 
  $\Aut(C)$ acts transitively on $C$ and leaves the distance partition
  invariant it follows that $\Gamma_4(\alpha)\cap C_4\neq\emptyset$.
  Thus, as $\Aut(C)_\alpha$ fixes $\Gamma_4(\alpha)$ setwise, we deduce
  that $\Aut(C)_\alpha$ has at least two orbits on $\Gamma_4(\alpha)$,
  and so it also has at least two orbits in its induced action on the
  set of $4$-subsets of $M$, which we denote by $M^{\{4\}}$.  The inner 
  product of the permutation characters for the actions of $M_{12}$ on
$M^{\{4\}}$
  and on the right cosets of $M_{11}$ is $2$ (see \cite[p.~33]{atlas}), and is
equal to the
  number of orbits of $\Aut(C)_\alpha\cong M_{11}$ on $M^{\{4\}}$.  Thus 
  $\Aut(C)_\alpha$ has exactly two orbits on $\Gamma_4(\alpha)$, and so acts
transitively on
  $\Gamma_4(\alpha)\cap C_4$.  Therefore Lemma \ref{method} implies
  that $\Aut(C)$ acts transitively on $C_4$.        
\end{proof}

\begin{proposition}\label{had12pun} The punctured Hadamard $12$ code
  $C^{(1)}$ is $\Aut(C^{(1)})$-completely transitive.  
\end{proposition}

\begin{proof}  Let $\Gamma=H(12,2)$ and $\Gamma'=H(11,2)$.  Recall from Section
\ref{prelim} that we can describe $C^{(1)}$
  as the projected code $\pi_J(C)$ where $J=M\backslash\{1\}$.
  Further recall the homomorphism $\chi$ from $\Aut(\Gamma)_1$
  onto $\Aut(H(J,2))\cong \Aut(\Gamma')$.  It is straight forward to check that 
  $\chi(\Aut(C))\cap\ker\chi=1$ and that $\chi(\Aut(C)_1)\leq\Aut(C^{(1)})$. 
Let $\alpha\in C$ be the zero codeword.  
  Since $\Aut(C)_\alpha\cong M_{11}$ acts transitively on $M$, it follows that
  $\Aut(C)_1$ is transitive on $C$, which implies that $\chi(\Aut(C)_1)$ is
transitive on $C^{(1)}$.    

  Let $\alpha'\in C^{(1)}$ be the zero codeword.  Then $\Aut(C^{(1)})_{\alpha'}$
contains $\chi(\Aut(C)_1\cap\Aut(C)_\alpha)\cong\PSL(2,11)$
  acting $2$-transitively on $J$.  Since $C^{(1)}$ has minimum distance 
  $\delta=5$, arguing as in the proof of Proposition 
  \ref{had12prop}, $\chi(\Aut(C)_1)$ is transitive on $C^{(1)}_i$ for $i=1, 2$;
each weight $3$ vertex in $\Gamma'$ lies in $C_2^{(1)}\cup C_3^{(1)}$; and $\Gamma'_3(\alpha')\cap
C_i^{(1)}$ is 
  non-empty for $i=2, 3$.  Thus $\Aut(C^{(1)})_{\alpha'}$ has at least two
orbits on $\Gamma_3'(\alpha')$,
  and hence at least two orbits in its action on $J^{\{3\}}$.  Then inner
product of the permutation
  characters for the actions of $M_{11}$ on $J^{\{3\}}$ and on the right cosets
of $\PSL(2,11)$ is $2$ 
  (see \cite[p.~18]{atlas}).  Hence $\Aut(C^{(1)})_{\alpha'}$ has two orbits on
$\Gamma'_3(\alpha')$, and
  so acts transitively on $\Gamma'_3(\alpha')\cap C_3^{(1)}$.  Thus Lemma
\ref{method} implies that 
  $\Aut(C^{(1)})$ acts transitively on $C^{(1)}_3$.       
\end{proof}

\section{Proof of Theorem \ref{mainchhad2}}

Let $C$ be a completely regular code in $H(m,2)$ with minimum
distance $\delta$ for $(m,\delta)=(12,6)$ or 
$(11,5)$.  Also, let $A(m,\delta)$ denote the maximum number of
codewords in any binary code of length $m$ with minimum distance 
$\delta$.  By \cite[App. A, Fig. 1]{macwil}, $A(11,5)=24$, and by
\cite[p.~42]{macwil}, $A(12,6)=A(11,5)=24$. Hence $|C|\leq 24$.  

  As automorphisms of the Hamming graph preserve the
  property of complete regularity, by replacing $C$ with an equivalent
  code if necessary, we may assume the zero vertex,
  $\alpha$, is a codeword in $C$. Since $C$ is completely regular with
  minimum distance $\delta$, it follows that
  $C(\delta)=\Gamma_\delta(\alpha)\cap C\neq\emptyset$.  Thus $C(\delta)$ forms 
  a $t-(m,\delta,\lambda)$ design for some $\lambda$, where
  $t=\lfloor\frac{\delta}{2}\rfloor$.  
  In both cases ($(m,\delta)=(12,6)$ or $(11,5)$), we deduce from (\ref{arith1})
  and (\ref{arith2}) that $2$ divides $\lambda$.  
  Consider the codewords of weight $\delta$ whose support
  contains $\{1,2,\ldots,t\}$.  Because $C$ has minimum distance
  $\delta$, it follows that the intersection of the supports of any
  two of these codewords is $\{1,\ldots,t\}$.  Consequently, a simple
  counting argument gives $\lambda\leq (m-t)/(\delta-t).$  In both
  cases we deduce that $\lambda\leq 3$, and so $\lambda=2$.      

  \underline{\bf{Case $(m,\delta)=(12,6)$:}}  As $C(6)$ forms a
  $3-(12,6,2)$ design, it follows that $|C(6)|=22$.  This design is the extension
  of a symmetric $2-(11,5,2)$ design, and so, by \cite[p.11]{camvan},
  we deduce that for all $\beta\in C(6)$, the complement $\beta^c\in C(6)$.  Thus, because $C$ is completely regular,
  it follows that $C$ is antipodal, so $\alpha^c\in C$ and $|C|\geq 1+22+1=24$.  As $|C|\leq 24$, we
  conclude that $C$ consists exactly of $\alpha, \alpha^c$ and $22$ codewords of
  weight $6$ that form a $3-(12,6,2)$ design.  The same holds for the
  Hadamard $12$ code, and by Remark \ref{desrem}, this design is
  unique up to isomorphism.  Thus there exists an automorphism
  $\sigma\in L$ (the top group of $\Aut(\Gamma)$) such that $C^\sigma$
  is the Hadamard $12$ code.  This proves Theorem \ref{mainchhad2}
  (a).       
  
  \underline{\bf{Case $(m,\delta)=(11,5)$:}}  As $C(5)$ forms a
  $2-(11,5,2)$ design it follows that $|C(5)|=11$.  Let $\alpha_1,\alpha_2$ be
  the two codewords in $C(5)$ whose supports contain $\{1,2\}$.  Then $d(\alpha_1,\alpha_2)=6$, 
  and so, because $C$ is completely regular, it follows that $C(6)$ forms
  a $2-(11,6,\mu)$ design for some $\mu$.  We deduce from (\ref{arith1})
  and (\ref{arith2}) that $3$ divides $\mu$ and $11$ divides $|C(6)|$, and so, because $|C|\leq 24$, 
  $\mu=3$ and $|C(6)|=11$.  Consequently $|C|=23$ or $24$.
  If $|C|=23$, then in the distance distribution $a(C)$, $a_0=1$, $a_5=a_6=11$, and $a_i=0$ otherwise.
  However, if this holds then in the MacWilliams transform of $a(C)$ (see \ref{kracheqn}), 
  $a_2'=\Sigma_{i=0}^{11}a_iK_2(i)=-55$, contradicting \cite[Lemma 5.3.3]{vanlint}.  Thus $|C|=24$ 
  and there exists a unique $i\geq 7$ such that $a_i=1$.  Suppose $i\leq 10$.  Then $C(i)$ forms 
  a $2-(12,i,\mu')$ design, and by Fishers inequality \cite[1.14]{camvan}, $|C(i)|\geq 11$ which
  contradicts the fact that $|C|=24$.  Thus $a_{11}=1$, which implies that $C$ is antipodal.
  Thus $C$ consists exactly of $\alpha$, $\alpha^c$, the $11$ codewords of weight $5$ that form a $2-(11,5,2)$ design,
  and their complements, which form a $2-(11,5,3)$ design.  This is also true for the punctured Hadamard
  $12$ code, and because the $2-(11,5,2)$ design is unique up to isomorphism \cite{todd},
  there exists $\sigma\in L$ such that $C^\sigma$ is the punctured
  Hadamard $12$ code.  This proves Theorem \ref{mainchhad2} (b). 

  By Proposition \ref{had12prop} and Proposition \ref{had12pun}, the Hadamard
  $12$ code and punctured Hadamard $12$ code are both completely
  transitive.  Therefore, by \cite{fpa}, any code
  that is equivalent to either the Hadamard $12$ code or punctured
  Hadamard $12$ code is completely transitive.  Thus the final
  statement of Theorem \ref{mainchhad2} follows from parts (a) and (b).

\section[]{Acknowledgments} 
The authors would like to thank the anonymous reviewers
for their valuable comments and suggestions to improve the
quality of this paper. 

\def\cprime{$'$} \def\cprime{$'$}


\end{document}